\documentclass[a4paper]{amsart}

\usepackage[utf8]{inputenc}
\usepackage[T1]{fontenc}
\usepackage[english]{babel}

\usepackage[format=hang,font=small]{caption}
\usepackage{enumitem}
\usepackage{pgf,tikz,pgfplots}
\pgfplotsset{compat=1.8}
\usepackage{mathrsfs}
\usetikzlibrary{arrows}

\usepackage{xcolor}
\usepackage{amsmath,amsthm}
\usepackage{mathrsfs}
\usepackage{mathtools}
\usepackage{amsfonts,amssymb}

\usepackage{graphicx}

\usepackage[colorlinks=true]{hyperref}
\hypersetup{urlcolor=blue, citecolor=red}

\renewcommand{\phi}{\varphi}
\renewcommand{\theta}{\vartheta}
\renewcommand{\epsilon}{\varepsilon}

\newcommand{\R}{\mathbb{R}}
\newcommand{\N}{\mathbb{N}}

\newcommand{\EE}{\mathcal{E}}
\newcommand{\RR}{\mathcal{R}}
\newcommand{\JJ}{\mathcal{J}}
\newcommand{\BB}{\mathcal{B}}
\newcommand{\NN}{\mathcal{N}}
\newcommand{\UU}{\mathcal{U}}
\newcommand{\PP}{\mathcal{P}}
\newcommand{\GG}{\mathcal{G}}
\newcommand{\vm}{v_\mathrm{m}}

\newcommand{\lsh}{\ell_\mathrm{sh}}
\newcommand{\RRsh}{\RR_\mathrm{sh}}
\newcommand{\Csh}{C_\mathrm{sh}}
\newcommand{\JJsh}{\mathcal{J}_\mathrm{sh}}
\DeclareMathOperator{\dist}{dist}
\DeclareMathOperator{\graph}{graph}
\DeclareMathOperator{\Span}{span}
\DeclareMathOperator{\cone}{cone}

\providecommand{\abs}[1]{\left\lvert#1\right\rvert}               
\providecommand{\norm}[1]{\left\lVert#1\right\rVert}    
\newcommand{\scal}[2]{\left\langle #1,#2\right\rangle}

\newcommand*{\dd}{\mathop{}\!\mathrm{d}}

\newtheorem{theorem}{Theorem}
\newtheorem{lemma}[theorem]{Lemma}
\newtheorem{prop}[theorem]{Proposition}
\newtheorem{corollary}[theorem]{Corollary}
\theoremstyle{definition}
\newtheorem{defin}[theorem]{Definition}
\newtheorem{remark}[theorem]{Remark}

\numberwithin{equation}{section}

\title[Stabilization of periodic sweeping processes]{Stabilization of periodic sweeping processes
	and asymptotic average velocity for soft locomotors with dry friction}
\date{}
\author[Giovanni Colombo, Paolo Gidoni and Emilio Vilches]{}

\subjclass{Primary: 70K42, 34A60; Secondary: 34D45, 37C25.}
\keywords{Sweeping process, running-periodic solutions, relative-periodic solutions, asymptotic stability, crawling locomotion.}

\thanks{\textit{Email addresses}: colombo@math.unipd.it, gidoni@utia.cas.cz, emilio.vilches@uoh.cl}
\thanks{G.C. was partially supported by
	Padua University grant SID 2018 ``Controllability,
	stabilizability and infimum gaps for control systems'', prot. BIRD 187147, and by GNAMPA of INdAM. P.G. was partially supported by by the GA\v CR--FWF grant 19-29646L. E.V. was partially supported by ANID-Chile under grant Fondecyt de Iniciaci\'on N$^{\circ}$ 11180098.}

\thanks{$^*$ Corresponding author}

\begin{document}
	
\maketitle

\centerline{\scshape Giovanni Colombo}
\smallskip
{\footnotesize
	\centerline{Department of Mathematics ``Tullio Levi-Civita''}
	\centerline{University of Padua}
	\centerline{via Trieste 63, 35121 Padova (Italy)}
} 

\medskip

\centerline{\scshape Paolo Gidoni$^*$}
\smallskip
{\footnotesize
	\centerline{Institute of Information Theory and Automation}
	\centerline{Czech Academy of Sciences}
	\centerline{Pod vodárenskou veží 4, CZ-182 08   Praha 8 (Czech Republic)}
}

\medskip

\centerline{\scshape Emilio Vilches}
\smallskip
{\footnotesize
	\centerline{Instituto de Ciencias de la Ingenier\'ia}
	\centerline{Universidad de O'Higgins}
	\centerline{Av. Libertador Bernardo O’Higgins 611, 2841959  Rancagua (Chile)}
}
\medskip

\begin{abstract}	
	We study the asymptotic stability of periodic solutions for sweeping processes defined by a polyhedron with translationally moving faces. Previous results are improved by obtaining a stronger $W^{1,2}$ convergence. Then we present an application to a model of crawling locomotion. Our stronger convergence allows us to prove the stabilization of the system to a running-periodic (or derivo-periodic, or relative-periodic) solution and the well-posedness of an average asymptotic velocity depending only on the gait adopted by the crawler. Finally, we discuss some examples of finite-time versus asymptotic-only convergence.
\end{abstract}

\section{Introduction}

Biological  locomotion is usually described  by recognising periodic patterns, or \emph{gaits}, in the movement of limbs or other body parts: the flapping of a fin for a fish,  the movement of the legs during a stride for a walking o running animal, the peristaltic wave for an earthworm.
Periodicity is however only an ideal regime. For instance, a standing start may require a transient phase until periodicity is reached; similarly an obstacle along the path may produce a deviation from the ideal periodic pattern, so that another transient is necessary to recover it.

This capability to reach and restore a periodic behaviour is a key feature to be reproduced in bio-mimetic robotic locomotors.  Firstly, for control and optimization purposes, it allows to focus directly on the limit cycle, neglecting the transient phase phase. Secondly, it guarantees the stability under perturbations of the locomotion strategy.

There are several ways to obtain such a convergence to periodicity.
One option is to actively enforce this property, building a suitable feedback-loop. In other situations, however, such behaviour is a spontaneous property of the systems. Namely, it is sufficient for the locomotor to keep repeating the same periodic actuation and the evolution of the system will asymptotically convergence to a limit cycle (which depends on the chosen input). Such (passive) stability is observed and investigated in several types of locomotion, such as walking \cite{HoWi}, passive walking \cite{Mak},  crawling with viscous friction \cite{EldJac} and swimming \cite{PoFeTa}.

In this paper,  we discuss the asymptotic stability of periodic solutions for the models of (soft) crawlers extensively discussed in \cite{Gid18} and illustrated in Figure \ref{fig:crawler}. The system can be briefly described as a chain of material points, each subject to a time-dependent dry friction, joined by elastic actuated links. As customarily in the modelling of crawling, the system is studied at the quasistatic limit; a mathematical discussion of this choice can be found in \cite{GidRiv}. 

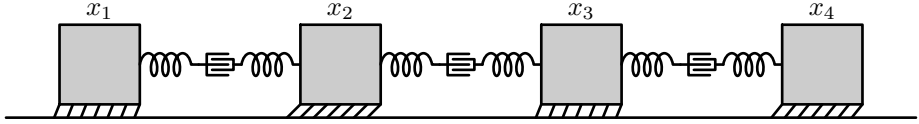
\begin{figure}[tb]
	\centering
	\begin{tikzpicture}[line cap=round,line join=round,x=4mm,y=4mm, line width=1pt, scale=0.88]
	\clip(-3,-2) rectangle (33,7);
	\draw [line width=1pt, fill=gray!40] (0,0.5)-- (3,0.5)--(3,3.5)-- (0,3.5)-- (0,0.5);
	\draw[decoration={aspect=0.5, segment length=1.5mm, amplitude=1.5mm,coil},decorate] (3,2)-- (5.2,2.);
	\draw[decoration={aspect=0.5, segment length=1.5mm, amplitude=1.5mm,coil},decorate] (6.8,2)-- (9.,2.);
	\draw (5.2,2)--(6.3,2);
	\draw (6.5,2)--(6.8,2);
	\draw (6.3,1.6)-- (5.5,1.6)--(5.5,2.4)--(6.3,2.4);
	\draw (5.7,1.8)-- (6.5,1.8)--(6.5,2.2)--(5.7,2.2);
	\draw [line width=1pt, fill=gray!40] (9,3.5)-- (9,0.5)-- (12,0.5)-- (12,3.5)-- (9,3.5);
	\draw[decoration={aspect=0.5, segment length=1.5mm, amplitude=1.5mm,coil},decorate] (12,2)-- (14.2,2.);
	\draw[decoration={aspect=0.5, segment length=1.5mm, amplitude=1.5mm,coil},decorate] (15.8,2)-- (18.,2.);
	\draw (14.2,2)--(15.3,2);
	\draw (15.5,2)--(15.8,2);
	\draw (15.3,1.6)-- (14.5,1.6)--(14.5,2.4)--(15.3,2.4);
	\draw (14.7,1.8)-- (15.5,1.8)--(15.5,2.2)--(14.7,2.2);
	\draw [line width=1pt, fill=gray!40] (18,3.5)-- (18,0.5)-- (21,0.5)-- (21,3.5)-- (18,3.5);
	\draw[decoration={aspect=0.5, segment length=1.5mm, amplitude=1.5mm,coil},decorate] (21,2)-- (23.2,2.);
	\draw[decoration={aspect=0.5, segment length=1.5mm, amplitude=1.5mm,coil},decorate] (24.8,2)-- (27.,2.);
	\draw (23.2,2)--(24.3,2);
	\draw (24.5,2)--(24.8,2);
	\draw (24.3,1.6)-- (23.5,1.6)--(23.5,2.4)--(24.3,2.4);
	\draw (23.7,1.8)-- (24.5,1.8)--(24.5,2.2)--(23.7,2.2);
	\draw [line width=1pt, fill=gray!40] (27,3.5)-- (27,0.5)-- (30,0.5)-- (30,3.5)-- (27,3.5);
	\draw [] (-2,0)-- (32,0);	
	\draw [] (0,0.5)-- (-0.2,0);
	\draw [] (0.5,0.5)-- (0.3,0);
	\draw [] (1,0.5)-- (0.8,0);
	\draw [] (1.5,0.5)-- (1.3,0);
	\draw [] (2,0.5)-- (1.8,0);
	\draw [] (2.5,0.5)-- (2.3,0);
	\draw [] (3,0.5)-- (2.8,0);
	\draw [] (9,0.5)-- (8.5,0);
	\draw [] (9.5,0.5)-- (9.,0);
	\draw [] (10,0.5)-- (9.5,0);
	\draw [] (10.5,0.5)-- (10.,0);
	\draw [] (11,0.5)-- (10.5,0);
	\draw [] (11.5,0.5)-- (11.,0);
	\draw [] (12,0.5)-- (11.5,0);
	\draw [] (18,0.5)-- (17.8,0);
	\draw [] (18.5,0.5)-- (18.3,0);
	\draw [] (19,0.5)-- (18.8,0);
	\draw [] (19.5,0.5)-- (19.3,0);
	\draw [] (20,0.5)-- (19.8,0);
	\draw [] (20.5,0.5)-- (20.3,0);
	\draw [] (21,0.5)-- (20.8,0);
	\draw [] (27,0.5)-- (26.6,0);
	\draw [] (27.5,0.5)-- (27.1,0);
	\draw [] (28,0.5)-- (27.6,0);
	\draw [] (28.5,0.5)-- (28.1,0);
	\draw [] (29,0.5)-- (28.6,0);
	\draw [] (29.5,0.5)-- (29.1,0);
	\draw [] (30,0.5)-- (29.6,0);
	\draw (1.5,4) node[] {$x_1$};
	\draw (10.5,4) node[] {$x_2$};
	\draw (19.5,4) node[] {$x_3$};
	\draw (28.5,4) node[] {$x_4$};
\end{tikzpicture}	
	
\caption{The model of soft crawler.}
\label{fig:crawler}
\end{figure}

First of all, let us clarify the form in which periodicity appears in locomotion: since 
a locomotor (hopefully) advances, clearly its evolution is not periodical in the canonical sense.
Let $\xi$ represent a generic body point of the locomotor in the reference configuration and denote with $x(\xi;t)$ its position at time $t$ in the deformed configuration. Periodicity in locomotion usually emerge in the form
\begin{equation}\label{eq:runningper}
\bar x(\xi;t)=\bar x_0+(t-t_0)\bar v_0+\bar p(\xi,t)
\end{equation}
where $\bar x_0$ and $\bar v_0$ are respectively the initial position and the \emph{average} velocity of a reference point of the locomotor (e.g. its barycentre), whereas the function $\bar p(\xi,\cdot)$  is periodic in time for every $\xi$.
Functions with the structure of $\bar x(\xi,\cdot)$ in \eqref{eq:runningper} are sometimes called \emph{derivo-periodic} \cite{And,AnBePa} (referring to the fact that they are primitives of periodic functions) or \emph{running-periodic} \cite{LeHoMi,Mar}. Moreover, they can be identified as a class within the more general family of \emph{relative-periodic} functions \cite{FaPaZo}, noticing that they can be decomposed as a periodic change in shape plus a translation in the position of the locomotor. This decomposition is classical in the modelling of locomotion \cite{KeMu} and will be crucial also in our paper, as we will soon show discussing equation \eqref{eq:sweep_form_intro}.

Let us now fix a gait $\GG$, which we identify with a prescribed actuation, with period $T$. Our aim is to prove the following two properties
\begin{itemize}
\item for every choice of admissible initial condition, the evolution of the system converge asymptotically to a running-periodic solution.
\item for every gait $\GG$, there is a unique velocity $v_0(\GG)$ such that all the running-periodic solutions compatible with the gait satisfy $\bar v_0=v_0(\GG)$.
\end{itemize}
These two results are the basis for the study of optimal gaits for our models, as it has been recently proposed in \cite{ColGid}.

Let us discuss more in detail the mathematical structure of our model.
As we show in Sec.~\ref{sec:model}, it is possible to reformulate the dynamics in the form
\begin{equation}\label{eq:sweep_form_intro}
\begin{cases}
\dot w\in -\NN_{K(t)}(w) \\
\dot y= \vm(t,-k^{-1} \dot w)
\end{cases}
\end{equation} 
where $\NN_C$ denotes the normal cone with respect to a convex set $C$.
The variable $w\in\R^{N-1}$ indirectly describes the shape of the locomotor in a frame of reference solidal with the crawler; the variable $y\in\R$ is the coordinate of the barycentre of the locomotor. Hence we would like to show that $w$ converges to a periodic function, whereas $y$ converges to a derivo-periodic function and
\begin{equation*}
\lim_{n\to+\infty}\frac{y(t_0+(n+1)T)-y(t_0+nT)}{T}=v_0(\GG)
\end{equation*} 
independently from the initial conditions. The properties of the crawler and of the gait are codified in the set $K(t)$, which is a polyhedron, and in the function $\vm$, whose values are defined as the minimizers of a dissipation functional on a certain subspace. 

The key technical contribution of the paper concerns therefore the asymptotic stability of the set of periodic solutions for dynamics of the form
\begin{equation} \label{eq:sweep_intro}
\dot z\in -\NN_{C(t)}(z)
\end{equation}
Differential inclusion of this form are known as \emph{Moreau's sweeping process}.
The existence and stability properties of periodic solution for \eqref{eq:sweep_intro}, when the set $C(t)$ changes periodically in time, has been first studied in \cite{Kre}, for the special case of a convex moving set $C(t)=C+c(t)$ in $\R^N$. In this framework, the existence of at least a periodic solution and the  global asymptotic stability of the set of periodic solution has been proven. 
These results has been recently generalized by \cite{GudMak} to the case of the finite intersection of convex moving sets. 

For our problem \eqref{eq:sweep_form_intro}, the set $K(t)$ is a polyhedron with translationally moving facets, hence included in the framework of \cite{GudMak}. However, as can be easily foreseen, the global asymptotic stability of the set of periodic solutions is not sufficient to our purposes, since it does not provide enough information on the behaviour of the derivative $\dot w$. Indeed, to show the well-posedness of the asymptotic velocity $v_0(\GG)$ we need the following two facts. First, as already shown in \cite{GudMak}, that all periodic solutions of \eqref{eq:sweep_intro} have the same derivative. Second, that the convergence to a periodic solution holds in a $W^{1,2}$-sense, compared with the $L^\infty$-convergence provided by the asymptotic stability in \cite{GudMak}.
More precisely, with Theorem \ref{th:attractor} we show that, for every solution $z(t)$ of \eqref{eq:sweep_intro}, where $C(t)$ is a polyhedron with translationally moving facets, there exists a periodic function $\bar z(t)$ such that
\begin{equation*}
	\lim_{q\to\infty} \| z(\cdot + q T)-\bar{z}(\cdot)\|_{W^{1,2}([t_0,t_0+T];\mathbb{R}^n)}=0.
\end{equation*}
where $q$ is an integer. This results can be then be applied to the second equation in \eqref{eq:sweep_form_intro}, to obtain $v_0(\GG)$.  Notice that, in addition to locomotion, this stronger convergence can be applied also to other contexts, for instance to improve the results for networks of elastoplastic springs in \cite{GudMak}.

\medbreak
A related problem is whether the convergence to a periodic behaviour occurs in finite time, or only asymptotically. Indeed, since our problem can be described by a sweeping process, uniqueness of solution is guaranteed in the future, but not in the past, so that periodicity may be reached exactly after a finite time. For our family of models, 
 finite-time convergence was already observed in \cite{GidDeS17}, where a few special gaits have been explicitly studied, noticing that a periodic behaviour was reached after just a single period.

It is not however easy to produce some general necessary condition for convergences in finite time. Some result has been proposed in \cite{GKMV} for the case of a moving polygon $C(t)=C+c(t)$, showing that under some conditions on $c(t)$, intuitively corresponding to a large continuous movement in a suitable direction, we have convergence within the first period. Such results have been recently generalized to higher dimension in \cite{GuMaRa}. Notice however that it is sufficient to replace the polygon $C$ with a sphere (o more generally a curved smooth boundary) to easily obtain a counterexample with asymptotic-only stability of a periodic solution \cite{GKMV}. 

In Section \ref{sec:finite} we provide some new results on this topic.
Firstly, with Theorem \ref{cor:cell} we show that if the set $C(t)$ is the Cartesian product of periodic intervals $[a_i(t),b_i(t)]$ we have convergence to a periodic solution within the first period.  Then, we provide two counterexamples showing asymptotic-only stability of the periodic solution for a moving polygon $C(t)=C+c(t)$.  The first one shows how, with a suitable choice of $c(t)$, it is always possible to construct examples of asymptotic-only convergence whenever $C$ has at least one acute angle. The second one complements the results in \cite{GKMV}, showing that, in special situations, also for a polyhedron a large continuous movement in a certain direction may not be sufficient.

%%%%%%%%%%%%%%%%%%%%%%%%%%%%%%%%%%%%%%%%%%%%%%%%%%%%%%%%

\section{Notation and basic definitions}
The open ball centered at $x\in\mathbb{R}^n$ and radius $r>0$ is denoted by $\BB(x,r)$. The normal cone of convex analysis to a convex set $C$
at $x\in C$ is denoted by $\NN_C(x)$. The distance of $x$ from a set $C$ is $\dist(x,C):=\inf \{ |x-y| : y\in C\}$.  Given a function $F(t,x)$, convex in the variable $x$, we denote with $\partial_x F$ the subdifferential with respect to the second variable in the classical sense of convex analysis. 
  
We consider set-valued maps $t\mapsto C(t)$ satisfying the following assumptions.

\begin{enumerate}[label=\textup{(C\arabic*)}]
\item \label{hyp:C1} At every time $t\in \R$, the set $C(t)$ is  nonempty. Moreover, there exist $b_1, \dots, b_m$ vectors in $\R^n$, and $c_1,\dots c_m, \colon \R\to \R$  Lipschitz continuous, $T$-periodic functions such that 
\begin{equation*}\label{eq:polyhed}
	C(t):=\{z\in \R^n: \scal{b_i}{z}\leq c_i(t) \quad \text{for every $i=1,\dots , m$} \}.
\end{equation*}
\item \label{hyp:C2} At every time $t\in \R$, the set $C(t)$ is compact.
\end{enumerate}

Notice that \ref{hyp:C1} implies in particular that $t\mapsto C(t)$ is $T$-periodic and Lipschitz continuous (with respect to the Hausdorff distance). Moreover, its graph on an interval $[t_0,t_1]$, namely $\graph (C, [t_0,t_1]):=\{ (t,x):
x\in C(t), t\in [t_0,t_1] \}$, is closed.

For every state $z\in C(t)$, we define the active set of constraints for the polyhedron at time $t$ as
\begin{equation}\label{eq:activeconstr}
	J(t,z)=\bigl\{i\in\{1,\dots,m\}: \scal{b_i}{z}= c_i(t) \bigr\},
\end{equation}
where we understand that $J(t,z)=\emptyset$ if $z$ belongs to the interior of $C(t)$.
Together with the set of active constraints at a point $z\in C(t)$, we consider also the set of ``active faces''
at time $t$, namely
\begin{equation}\label{def_faces}
	\mathcal{F}(t):=\{J(t,z): z\in C(t)\},
	\end{equation}
that is the set of all subsets of constraints that are active at some $z\in C(t)$. Using $\mathcal{F}(t)$, we define also define the following property characterizing the evolution of $C(t)$.

\begin{defin}[LPAF] We say that the property of \emph{local permanence of the active faces}, briefly LPAF, holds at time $t$ if there exists a neighbourhood
	$\UU_t$ of $t$ such that for every $s\in \UU_t$ we have $\mathcal{F}(s)=\mathcal{F}(t)$.
\end{defin}
Of course the polyhedron $C(t)$ can be seen as the intersection of $m$ halfplanes translating in time, namely
\begin{equation}\label{eq:eqrepr}
C(t)=\bigcap_{i=1}^m \big(\{ z\in\mathbb{R}^n : \langle b_i,z\rangle \le 0\} + c_i(t) b_i\big)=
:\bigcap_{i=1}^m \left( C_i^0 + c_i(t) b_i\right),
\end{equation}
where $C_i^0$ denotes the closed halfplane $\langle b_i,z\rangle \le 0$.
Some constraint qualification assumptions will be supposed on $C(t)$. 

\begin{defin}[LICQ] We say that the \emph{linear-independence constraint-qualification}, briefly LICQ,
	is valid at $t$ if for each $z\in C(t)$ the family of vectors $b_i,\, i\in J(t,z)$, is linearly independent.
\end{defin}
We observe that the validity of LICQ at $t\in \R$ implies that there exists a constant $\gamma >0$, independent of 
$z\in C(t)$, such that
\begin{equation}\label{eq:revtrineq}
\sum_{i\in J(t,z)} \lambda_i |b_i| \le \gamma \left|  \sum_{i\in J(t,z)} \lambda_i b_i  \right|\quad\text{for all } \lambda_i\ge 0.
\end{equation}
Moreover, it is easy to see that, for each $t\in \R$, (C1) and the validity of LICQ at $t$ imply LPAF at $t$.

\smallskip

We introduce two more assumptions on the regularity of the shape of $C(t)$.
\begin{enumerate}[start=3,label=\textup{(C\arabic*)}]
\item \label{hyp:C3} At almost every time $t\in \R$, the set $C(t)$ satisfies LICQ.
\item \label{hyp:C4} At almost every time $t\in \R$, the set $C(t)$ satisfies LPAF; in other words, the set 
$\mathcal{F}(t)$, that was defined at \eqref{def_faces}, is locally constant on an open set of full measure.
\end{enumerate}
It is easy to see that (C3) implies (C4), while it may happen that (C1) and (C2) are valid,
but (C4) fails. For example, consider a trapezoid that, thanks to a translation of one edge, 
becomes a triangle on a Cantor set with positive measure.

\medbreak

Our investigation will be devoted to the $T$-periodic polyhedral sweeping process
\begin{equation}\label{eq:sweepingproc}
\dot z \in-\NN_{C(t)}(z)
\end{equation}
 A solution of \eqref{eq:sweepingproc} is an absolutely continuous function $z(\cdot)$, defined on an interval $I$, that
satisfies the equation a.e. It is well known (see, e.g., \cite{CT}), that \eqref{eq:sweepingproc} together with an admissible initial condition $z(t_0)=z_0\in C(t_0)$
admits one and only one \textit{forward-in-time} solution.

\section{Asymptotic behavior for a class of periodic sweeping processes} \label{sec:theory}

The following is the main result of this section. The initial time $t_0$ is fixed once for all.
\begin{theorem}\label{th:attractor}
Consider the $T$-periodic polyhedral sweeping process \eqref{eq:sweepingproc}, where \ref{hyp:C1},
\ref{hyp:C2} and \ref{hyp:C4} hold.
Let $z\colon [t_0,+\infty)\to\mathbb{R}^n$ be a solution of \eqref{eq:sweepingproc}. Then there exists a $T$-periodic solution $\bar{z}$
of \eqref{eq:sweepingproc}  such that
\begin{equation}\label{eq:pente}
\lim_{q\to\infty} \| z(\cdot + q T)-\bar{z}(\cdot)\|_{W^{1,2}([t_0,t_0+T];\mathbb{R}^n)}=0.
\end{equation}
\end{theorem}
\begin{remark} \label{rem:gudmak}
 Theorem \ref{th:attractor} improves \cite[Theorem 4.3]{GudMak} (cf.~Theorem \ref{th:gudmak} below), where essentially the same convergence result is proved under the coarser topology of $L^\infty([t_0,t_0+T];\mathbb{R}^n)$. As we show in Section \ref{sec:model}, a stronger convergence is required for our application, namely the $L^1$-convergence of the derivatives.  In order to prove Theorem \ref{th:attractor}, we will employ the fact that \eqref{eq:sweepingproc} can be locally reduced 
to a play operator (i.e., to a case where the moving set is a pure translation), cf. Lemma~\ref{lem:red_play}, combined with ideas taken 
from the proof of Theorem 3.12 in \cite{Kre}, that allow to obtain the
strong convergence in $L^2$ of the derivatives.
\end{remark}

\medskip

Before to demonstrate the proof of Theorem \ref{th:attractor}, we present three technical lemmas.

\begin{lemma} \label{lem:indep_base}
	
	Assume \ref{hyp:C1}, fix $t\in [t_0,t_0+T]$ and assume that \textup{LICQ} holds for $C(t)$. Let $v\in\R^n$ and let $z\in C(t)$ be such that $v\in\NN_{C(t)}(z)$. Then there exists a unique choice of nonnegative numbers
	$\lambda_i$, $i=1,\ldots,m$, such that
	\begin{equation} \label{eq:reprv}
	v=\sum_{i=1}^m \lambda_i b_i,\;\text{ where }\lambda_i>0\Rightarrow i\in J(t,z).
	\end{equation}
	Moreover, the coefficients $\lambda_i$ depend only on $v$ and on the set of active faces 
	$\mathcal{F}(t)$ defined in \eqref{def_faces};
	in particular they are independent of the base point $z$, and of changes in $C(t)$ that do not create new faces, or remove existing ones.
Finally the coefficients $\lambda_i$ are Lipschitz continuous with respect to $v$.
\end{lemma}
\begin{proof}
	First of all, we recall that, since $C(t)$ is a convex polyhedron, we have 
	\begin{equation*}
	\NN_{C(t)}(z)=\cone\{b_i, i\in J(t,z)\}
	\end{equation*}
	The first part of the lemma follows directly from this and LICQ. 
	
	We then observe that the set of the points $x\in C(t)$ such that $v\in \NN_{C(t)}(x)$ is exactly a face of $C(t)$. Since the active constraints on the interior of the face are active on the whole face, by the first part of the lemma we deduce that the representation \eqref{eq:reprv} is independent of the base point $z$. 
	
	For $v$ varying in the normal cone to the relative interior of each face, the coefficients $\lambda$'s
	that appear in the representation \eqref{eq:reprv} are, by LICQ, the solution of a linear system with maximal rank (that is equal to the number of active
	constraints). Thus they are a Lipschitz, actually, a linear, function of the datum $v$. Moreover, the boundary of each normal cone matches
	with the normal cone at the corresponding neighbouring faces. 
	Hence, since there is only a finite number of faces, the mapping $v\mapsto \lambda (v)$ is piecewise linear, with a finite number of linear components, each active exactly on a cone. In particular, this implies that it is globally Lipschitz.
\end{proof}

\begin{lemma} \label{lem:meas_sel}  Let $C(t)$ satisfy \ref{hyp:C1} and \ref{hyp:C3}.
 Let $z\colon[t_0,+\infty)\to \mathbb{R}^n$ be a solution of \eqref{eq:sweepingproc}. Then there exists a unique collection of nonnegative measurable
functions $\lambda_i$, $i=1,\ldots ,m$, satisfying, for a.e.~$t\in [t_0,+\infty)$,
\begin{equation}\label{eq:zdot}
\dot{z}(t)=-\sum_{i=1}^m \lambda_i (t) b_i
\end{equation}
and
\begin{equation}\label{eq:lambda_pos}
\lambda_i(t)>0 \Rightarrow i\in J(t,z(t)).
\end{equation}
\end{lemma}

\begin{proof}
Using Lemma \ref{lem:indep_base}, we construct for a.e.~$t\in [t_0,+\infty)$ and every $v\in\R^n$,
the unique nonnegative maps $\tilde\lambda_i(t,v)$. By  \ref{hyp:C1}, we deduce that the map 
$t\mapsto\mathcal{F}(t)$, defined according to \eqref{def_faces}, is measurable, 
hence also $t\mapsto\tilde\lambda_i(t,v)$ is measurable for every 
$v\in \R$. Since $v\mapsto\tilde\lambda_i(t,v)$ is Lipschitz continuous at almost every $t$ and $z(t)$ is 
absolutely continuous, we deduce that the maps 
$t\to \lambda_i(t):=\tilde\lambda_i(t,\dot z(t))$ are measurable. 
\end{proof}

\begin{lemma} \label{lem:red_play}
Let $C(t)$ satisfy \ref{hyp:C1} and \ref{hyp:C4}. 
Then the sweeping process \eqref{eq:sweepingproc} is locally of 
play-type. More precisely, for a.e.~$t\in [t_0,t_0+T]$ and all
$x\in C(t)$ there exist $\delta_1>0$, $\delta_2>0$, a Lipschitz function 
$v\colon(t-\delta_1,t+\delta_1)\to \mathbb{R}^n$,
and a polyhedron $\mathfrak{C}=\mathfrak{C}(x)$ such that, for each
$(\tau,y)\in (t-\delta_1,t+\delta_1)\times \BB(x,\delta_2)$ the solution of
\begin{equation*}
\left\{
\begin{aligned}
\dot{z}(\cdot)&\in -\NN_{C(\cdot)} (z(\cdot))\\
z(\tau)&=y
\end{aligned}
\right.
\end{equation*}
(that exists if and only if $y\in C(\tau)$) is also a solution of
\begin{equation*}
\dot{z}(\cdot) \in -\NN_{\mathfrak{C}+v(\cdot)}(z(\cdot))\;\text{ a.e.~in } (t-\delta_1,t+\delta_1).
\end{equation*}
\end{lemma}
\begin{proof}
Fix $t\in [t_0,t_0+T]$  and $x\in C(t)$, such that LPAF holds at $t$. Since the $c_i(\cdot)$ are continuous, there exist $\bar{\delta}_1,\bar{\delta}_2>0$ such that
for all $(\tau,y)\in B:= [t-\bar{\delta}_1,t+\bar{\delta}_1]\times \BB(x,\bar{\delta}_2)$ one has 
$J(\tau,y)\subseteq J(t,x)$; moreover, we choose $\bar{\delta}_1$ sufficiently small to assure $[t-\bar{\delta}_1,t+\bar{\delta}_1]\subset \UU_t$, where $\UU_t$ is the neighbourhood in the definition of LPAF. For a set of
indices $\Sigma\subset \{ 1,\ldots ,m\}$, define
\begin{equation*}
\tilde{C}(\Sigma,t) := \bigcap_{i\in\Sigma}\big\{ z\in\mathbb{R}^n : \langle b_i,z\rangle \le c_i(t) \big\}
\end{equation*}
and observe that for each solution $z:[t-\bar{\delta}_1,t+\bar{\delta}_1]\to\mathbb{R}^n$ of 
\eqref{eq:sweepingproc} such that
$|z(\tau)-x|<\bar{\delta}_2$ for all $\tau\in [t-\bar{\delta}_1,t+\bar{\delta}_1]$ one has
\begin{equation*}
\dot{z}(\tau)\in - \NN_{\tilde{C}(J(t,x),\tau)}(z(\tau))\;\text{ a.e.~on }[t-\bar{\delta}_1,t+\bar{\delta}_1],
\end{equation*}
because $C(\tau)\cap \BB(x,\bar{\delta}_2)=\tilde{C}(J(t,x),\tau)\cap \BB(x,\bar{\delta}_2)$, 
provided $|\tau-t| \le  \bar{\delta}_1$, 
so that $\NN_{C(\tau)}(z(\tau))=\NN_{\tilde{C}(J(t,x),\tau)}(z(\tau))$ for all
$\tau\in [t-\bar{\delta}_1,t+\bar{\delta}_1]$. Without loss of generality assume that $J(t,x)=\{ 1,\ldots , n(t,x)\}$ for a suitable $n(t,x)\geq 0$ and consider, for each $\tau\in (t-\bar\delta_1,t+\bar\delta_1)$, the system
\begin{equation*}
	\begin{cases}
		\langle b_1,x+v \rangle &= c_1(\tau)\\
		&\vdots\\
		\langle b_{n(t,x)}, x+v\rangle & = c_{n(t,x)}(\tau)
	\end{cases}
\end{equation*}
in the unknown $v$.	Even when the vectors $b_1,\dots, b_{n(t,x)}$ are not linearly independent, since LPAF holds at $t$ we know that the system admits at least one solution on $(t-\bar\delta_1,t+\bar\delta_1)$. Moreover, if we add the constraint  $v\in \Span \{b_1,\dots, b_{n(t,x)}\}$ the system has a \emph{unique} solution, which we call $v(\tau)$. We notice that $v(\tau)$ is a  Lipschitz-continuous function of $\tau$ and satisfies $v(t)=0$.

We claim now that 
$\tilde{C}(J(t,x),t)+v(\tau)=\tilde{C}(J(t,x),\tau)$. Indeed,
let $y=z+v(\tau)$, with $z\in\tilde{C}(J(t,x),t)$. Then, for $i=1,\ldots,n(t,x)$,
\begin{equation*}
\begin{split}
\langle b_i,y\rangle = \langle b_i,z+v(\tau)\rangle &\leq c_i(t)+\langle b_i,v(\tau)\rangle\\
&=c_i(t)+\langle b_i,x+v(\tau)\rangle-\langle b_i,x\rangle\\
&=c_i(t)+c_i(\tau)-c_i(t)=c_i(\tau).
\end{split}
\end{equation*}
Conversely, let $y\in\tilde{C}(J(t,x),\tau)$. Then, for $i=1,\ldots ,n(t,x)$,
\begin{equation*}
\begin{split}
\langle b_i, y-v(\tau)\rangle &= \langle b_i,y-(x+v(\tau))\rangle + \langle b_i,x\rangle\\
&\le c_i(\tau) - c_i(\tau) + c_i(t)=c_i(t),
\end{split}
\end{equation*}
that confirms the claim. 

Let now $\delta_2>0$
and $0<\delta_1<\bar{\delta}_1$ be sufficiently small so that all solutions $z$ of
$\dot{z}\in - \NN_{C(t)}(z)$ with $z(\tau)=y\in B(x,\delta_2)$ satisfy $|z(s)-x|<\frac{\bar{\delta}_2}{2}$ 
for all $s\in (t-\delta_1,t+\delta_1)$. Setting $\mathfrak{C}:=\tilde{C}(J(t,x),t)$ we conclude the proof.
\end{proof}

In order to prove Theorem \ref{th:attractor}, let us first recall some results from \cite{GudMak}.  Let $\mathcal Z$ be the set of $T$-periodic solutions of 
\eqref{eq:sweepingproc} and $\mathcal Z(t):=\{ z(t):z\in \mathcal Z\}$. 

\begin{theorem}[cf.~Theorem 4.3 in \cite{GudMak}] \label{th:gudmak} Let $C(t)$ satisfy \ref{hyp:C1}. Then the set $\mathcal{Z}(t)$  is closed and convex for every $t\in [t_0,t_0+T]$. Moreover, given two $T$-periodic solutions $\bar z,\hat z \in \mathcal{Z}$  of  \eqref{eq:sweepingproc}, 
then
\begin{equation}  \label{eq:eqder2}
	\dot{\bar z}(t)=\dot{\hat z}(t) \text{ for almost every $t\in[t_0,t_0+T]$}.
\end{equation}	
Finally $\mathcal Z(t)$ is a global attractor for the system, meaning that any solution $z\colon[t_0,+\infty)\to\mathbb{R}^n$ of \eqref{eq:sweepingproc} satisfies
\begin{equation} \label{eq:convtoZ}
	\lim_{t\to+\infty} \dist(z(t),\mathcal{Z}(t))=0
\end{equation}
\end{theorem}
Let us remark that here we stated Theorem \ref{th:gudmak} under the assumption \ref{hyp:C1}, 
but it is actually proved in \cite{GudMak} considering a slightly more general framework. 
We also observe that given $\bar z\in\mathcal{Z}$, every other $T$-periodic solution $\tilde z\in \mathcal Z$ 
can be characterized as
\begin{equation}\label{eq:per_char}
\tilde z(t)=\bar z(t)-\bar z(t_0)+\tilde z(t_0)
\end{equation}
for any time $t_0\in \R$.

\begin{proof}[Proof of Theorem \ref{th:attractor}]
Let $z\colon [t_0,+\infty)\to\mathbb{R}^n$ be a solution of \eqref{eq:sweepingproc}. Our first step is to prove that there exists $\bar z\in \mathcal{Z}$ such that
\begin{equation*}
	\lim_{q\to\infty} \| z(\cdot + q T)-\bar{z}(\cdot)\|_{L^\infty([t_0,t_0+T];\mathbb{R}^n)}=0.
\end{equation*}
Let us start by noticing that,  by  \eqref{eq:convtoZ} in Theorem \ref{th:gudmak} we deduce that  for each $q\in\N$ there exists a periodic solution $z_q\in \mathcal Z$ such that
\begin{equation*}
\abs{z(t_0+qT)-z_q(t_0+qT)}=d(z(t_0+qT),\mathcal{Z}(t_0+qT))\to 0 
\end{equation*}
as $q\to\infty$. By compactness, due to \ref{hyp:C2}, there exist a subsequence of positive integers $\{ q_k\}$, and $\bar{z}\in X$ such that $z_{q_k}(t_0)\to \bar{z}(t_0)$. By \eqref{eq:per_char}, this implies that $z_{q_k}\to \bar{z}$
uniformly in $[t_0,t_0+T]$. Hence,
\begin{equation}\label{eis}
|z(t_0+q_kT)-\bar{z}(t_0)|=|z(t_0+q_kT)-\bar{z}(t_0+q_kT)|\to 0\;\text{ as }q_k\to\infty.
\end{equation}
We claim now that the whole sequence $z(\cdot + qT)$ converges to $\bar{z}(\cdot)=\bar z(\cdot+qT)$ uniformly in $[t_0,t_0+T]$ as $q\to+\infty$. Indeed, by convexity, we have
\begin{equation}\label{eq:monotony}
\frac{\dd}{\dd t} \frac12 | z(t)-\bar{z}(t)|^2 = \langle \dot{z} (t)-\dot{\bar{z}}(t),z(t)-\bar{z}(t)\rangle \le 0\quad\text{a.e. $t\in\R$}.
\end{equation}
This implies that, for every $t_1,t_2$ with $t_2>t_1\geq t_0$ we have
\begin{equation}
\abs{z(t_2)-\bar{z}(t_2)}\leq \abs{z(t_1)-\bar{z}(t_1)}.
\end{equation} 
We deduce immediately that the sequence $\abs{z(t_0+qT)-\bar{z}(t_0)}$ is monotone decreasing, hence from \eqref{eis} we obtain $|z(t_0+qT)-\bar{z}(t_0)|\to 0$ as $q\to +\infty$ for the whole sequence. It follows that
\begin{equation}
\norm{z(\cdot + q T)-\bar{z}(\cdot)}_{L^{\infty}([t_0,t_0+T];\mathbb{R}^n)}\leq \abs{z(t_0+qT)-\bar{z}(t_0)} \to 0.
\end{equation}
We wish to prove now that the convergence of $z(\cdot + qT)$ to $\bar{z}(\cdot)$ is indeed strong in $W^{1,2}([t_0,t_0+T];\mathbb{R}^n)$.
Applying Lemma \ref{lem:red_play} and using the compactness of $\graph(C,[t_0,t_0+T])$, we find finitely many points $(t_\ell,x_\ell)\in [t_0,t_0+T]\times
\mathbb{R}^n$, $\ell =1,\ldots ,\bar \ell$, with $x_\ell\in C(t_\ell)$, and positive numbers $\delta_1^\ell,\delta_2^\ell$ together with a Lipschitz function
$v_\ell: (t_\ell -\delta_1^\ell,t+\delta_1^\ell)\to\mathbb{R}^n$ and a polyhedron $\mathfrak C_\ell$ such that
\begin{equation*}
\bigcup_{\ell =1}^{\bar \ell} \big(t_\ell+\delta_1^\ell,t_\ell+\delta_1^\ell\big)\times \BB\Big(x_\ell,\frac{\delta_2^\ell}{2}\Big)\supset\graph(C,[t_0,t_0+T])
\end{equation*}
and for each $\ell =1,\ldots ,\bar \ell$ and each $(\tau,y)\in \big(t_\ell+\delta_1^\ell,t_\ell+\delta_1^\ell\big)\times \BB\big(x_\ell,\delta_2^\ell)$
the solution of 
\begin{equation*}
\begin{cases}
\dot{z}(\cdot)&\in - \NN_{C(\cdot)}(z(\cdot))\\
z(\tau)&=y
\end{cases}
\end{equation*}
is the solution of
\begin{equation*}
\begin{cases}
\dot{z}(\cdot)&\in - \NN_{\mathfrak{C}_\ell + v_\ell (\cdot)}(z(\cdot))\\
z(\tau)&=y
\end{cases}
\end{equation*}
in $(t_\ell-\delta_1^\ell,t_\ell+\delta_1^\ell)$. Hence the graph of $\bar{z}$ is contained in the union of the sets $(t_\ell-\delta_1^\ell, t_\ell+\delta_1^\ell)
\times \BB\big(x_\ell,\frac{\delta_2^\ell}{2}\big)$ with respect to a subcollection $\ell\in I_{\bar{z}}$. By uniform convergence, for all $q$ large enough also the graph 
of $z(\cdot + qT)$ is contained in the union of the same elements of the covering. Therefore, owing to Lemma  \ref{lem:red_play}, a.e.~in each interval 
$(t_\ell-\delta_1^\ell,t_\ell+\delta_1^\ell)$, $\ell\in I_{\bar{z}}$, we have both $\dot{z}(s+qT)\in - \NN_{\mathfrak C_\ell+v_\ell(s)}(z(s+qT))$ and
$\dot{\bar{z}}(s)\in - \NN_{\mathfrak C_\ell+v_\ell(s)}(\bar{z}(s))$. Fix now the interval $I_\ell:=(t_\ell-\delta_1^\ell,t_\ell+\delta_1^\ell)$, with
$\ell\in I_{\bar{z}}$. We know from the previous arguments that $z(\cdot + qT)$ converges to $\bar{z}(\cdot)$ uniformly in $I_\ell$.
Set now $w_q(\cdot):=z(\cdot+qT)$. Since $\bar{z}$ is $T$-periodic, we have that $w_q\to\bar{z}$ uniformly in $I_\ell$ and 
$\dot{\bar{z}}(s)\in -\NN_{\mathfrak{C}_\ell+v_\ell(s)}(\bar{z}(s))$ a.e.~in $I_\ell$. For each $q$ large enough and each $s\in I_\ell$, we have that
$w_q(s)-v_\ell(s)\in \mathfrak{C}_\ell$. Fix $s\in I_\ell$ such that $\dot{w}_q(s)$ and $\dot{v}_\ell(s)$ exist. Then
\begin{equation*}
\Big\langle\dot{w}_q(s),\frac{w_q(s+h)-v_\ell(s+h)-(w_q(s)+v_\ell(s))}{h}\Big\rangle \ge 0 \; \text{for all }h>0,
\end{equation*}
while
\begin{equation*}
\Big\langle\dot{w}_q(s),\frac{w_q(s+h)-v_\ell(s+h)-(w_q(s)+v_\ell(s))}{h}\Big\rangle \le 0 \; \text{for all }h<0.
\end{equation*}
Hence, by passing to the limit as $h\to 0$ we obtain that
\begin{equation}\label{dyo}
\langle \dot{w}_q(s),\dot{w}_q(s)-\dot{v}_\ell(s)\rangle =0.
\end{equation}
The same argument shows that we have as well
\begin{equation}   \label{treis}
\langle\dot{\bar{z}}(s),\dot{\bar{z}}(s)-\dot{v}_\ell(s)\rangle =0\;\text{ a.e.~in } I_\ell.
\end{equation}
By \eqref{dyo}, we deduce that $\norm{\dot{w}_q}_{L^2(I_\ell)}\leq \norm{\dot v_l}_{L^2(I_\ell)}$; hence, up to a subsequence, $\dot{w}_q$ convergence weakly in  $L^2((t_0,t_0+qT);\mathbb{R}^n)$. By a standard argument, using the fact that $w_q\to \bar z$ uniformly, we deduce that actually the whole sequence  $\dot{w}_q$ converges to $\dot{\bar{z}}$ weakly in $L^2((t_0,t_0+qT);\mathbb{R}^n)$,
hence in $L^2(I_\ell;\mathbb{R}^n)$. 

To complete the proof, namely to show that $\dot{w}_q$ converges to $\dot{\bar{z}}$ strongly in  $L^2(I_\ell;\mathbb{R}^n)$, it is therefore enough to show that $\|\dot{w}_q\|_{L^2(I_\ell)}^2\to\|\dot{\bar{z}}\|_{L^2(I_\ell)}^2$. But \eqref{dyo}
implies, thanks to the weak convergence, that
\begin{equation*}
\|\dot{w}_q\|_{L^2(I_\ell)}^2 =\int_{I_\ell} \langle\dot{w}_q(s),\dot{v}_\ell(s)\rangle\, \dd s \rightarrow
\int_{I_\ell} \langle\dot{\bar{z}}(s), \dot{v}_\ell(s)\rangle\, \dd s,
\end{equation*}
and the latter is equal to $\|\dot{\bar{z}}\|_{L^2(I_\ell)}^2$ by \eqref{treis}. Since the intervals $I_\ell$, $\ell\in I_{\bar{z}}$ cover the whole of
$[t_0,t_0+T]$, the proof of \eqref{eq:pente} is concluded. 
\end{proof}
\begin{corollary} \label{cor:meas_decomp}
	Under the assumptions of Theorem \ref{th:attractor}, with \ref{hyp:C4} replaced by \ref{hyp:C3}, 
let $\lambda_i(\cdot)$, resp. $\bar{\lambda}_i(\cdot)$ be the coefficients appearing in
	\eqref{eq:zdot} for $\dot{z}$, resp. for $\dot{\bar{z}}$. Then
	\begin{equation*}  
	\lim_{q\to\infty} \| \lambda_i (\cdot+qT)-\bar{\lambda}_i(\cdot)\|_{L^2((t_0,t_0+T);\mathbb{R}^n)}=0.
	\end{equation*}
\end{corollary} 
\begin{proof}
Observe that, by \ref{hyp:C3} and \eqref{eq:revtrineq}, for a.e.~$t\in [t_0,+\infty)$there exists 
$\gamma (t)>0$ such that
\begin{equation*}
\sum_{i=1}^m \lambda_i(t) |b_i| \le \gamma (t)
\left|  \sum_{i=1}^m \lambda_i(t) b_i  \right|.
\end{equation*}	
Since \ref{hyp:C3} implies \ref{hyp:C4}, the measurable function $\gamma (\cdot)$ can be supposed to be bounded,
which gives in turn the strong convergence in $L^2(t_0,t_0+T)$ of each sequence $\lambda_i(\cdot+qT)$ to a 
function $\tilde \lambda_i$. In order to show that $\tilde \lambda_i=\bar \lambda_i$ (as $L^2$ functions), 
by Lemma \ref{lem:meas_sel} it is sufficient to show that $\tilde \lambda_i(t)=0$ for almost every $t$ such that 
$i\notin J(t,\bar z)$.
This can be deduced by noticing that, as a consequence of Theorem \ref{th:attractor},  
for every $\hat t\in[t_0,t_0+T]$, there exist $\hat q\in \N$,  such that 
$J(\hat t,z(\hat t+qT))\subseteq J(\hat t, \bar z(\hat t))$ for every $q\geq \hat q$. 
Hence, for $q\to+\infty$, we have $\lambda_i(\cdot+qT)\to 0$  pointwise almost everywhere on 
$\{t:i \notin J(t,\bar z)\}$, so that the desired fact follows by the Dominated Convergence Theorem.

\end{proof}

\section{Asymptotic average velocity for soft crawlers} \label{sec:model}
Following \cite{Gid18}, we consider a general model of crawling locomotor and show how it can be formulated as a 
sweeping process, in order to apply the results of the previous section.

\subsubsection*{A model of crawler}

Our model of crawler, illustrated in Figure \ref{fig:crawler}, consists of a chain of $N\geq 2$ material points, so that its body 
is described by the set $\Omega_N=\{\xi_1,\xi_2,\dots,\xi_N\}\subset\R$. The displacement of the crawler is described by 
a vector $x=(x_1,x_2,\dots,x_N)$ in $X=\R^{N}$, identifying with $x_i$ the displacement of the point $\xi_i$.
We assume that the body of the crawler has $N-1$  links joining each couple of consecutive blocks. 
Each link is composed by an actuator, namely a segment with time dependent length $L_i(t)$, and an elastic spring, in series with the actuator.

As discussed in detail in \cite{GidRiv}, when the rate of the actuation is very slow, as is usually the case of  such locomotors, the evolution of the system is given by the force balance between the elastic forces 
produced by the deformation of the springs in the links and the frictional forces in the points of interaction with the surface, whereas inertial forces can be neglected, as well as possible viscous resistances in the system.
Such force balance can be expressed variationally in the following form
\begin{equation} \label{eq:forcebalance}
	0\in D_x\EE(t,x)+\partial_{\dot x} \RR(t,\dot x)
\end{equation}
Here  the internal energy $\EE(t,x)$ of the crawler is the sum of the $N-1$ terms $\EE_i(t,x)$ associated with each link, namely
\begin{equation}\label{eq:linkenergy}
\EE(t,x)=\sum_{i=1}^{N-1}\EE_i(t,x)=\sum_{i=1}^{N-1}\frac{k}{2}(x_{i+1}-x_i-L_i(t))^2
\end{equation}
where $k>0$ is the elastic constant of the springs, which we assume to obey Hooke's law. We assume that the actuations $L_i(t)$ are Lipschitz continuous.

Friction forces are represented by the dissipation potential
\begin{equation} \label{xeq:dissipoint}
\RR(t,\dot x)=\sum_{i=1}^{N} \RR_i(t,\dot x_i) \qquad \text{with}\quad  \RR_i(t,\dot x_i)=\begin{cases}
-\mu_i^-(t)\dot x_i  &  \text{if $\dot x_i\leq 0$} \\
\mu_i^+(t)\dot x_i &  \text{if $\dot x_i\geq 0$}
\end{cases}
\end{equation}
where the functions $\mu_i^\pm(t)$ are assumed to be Lipschitz continuous and such that there exists two 
positive constants $\alpha_1,\alpha_2$ for which $\alpha_1\leq \mu_i^\pm(t)\leq \alpha_2$
 for every index $i$ and direction $\pm$.

We notice that, since we are working in a quasistatic setting, not every initial configuration is admissible, 
since too large tensions in the elastic body cannot be counterbalanced by the  friction forces, which are bounded by construction.
Let us denote with $e_1,\dots,e_N$ the canonical base of $X$. We define the set $C(t)$ as
\begin{equation}\label{eq:defC}
C(t):=\{x\in X : -\mu_i^-(t)\leq\scal{e_i}{x}\leq \mu_i^+(t)\;\, \text{for $i=1,\dots,N$}\}=\partial_{\dot x}\RR(t,0)
\end{equation}
We notice that since $\RR$ is positively homogeneous of degree one, we have 
$\partial_{\dot x}\RR(t,\dot{x})\subseteq \partial_{\dot x}\RR(t,0)$ for every $\dot x\in X$. We say that an initial state $x(t_0)=x_0$ 
is \emph{admissible} for problem \eqref{eq:forcebalance} if it satisfies
\begin{equation}
	-D_x\EE(t,x_0)\in C(t_0)
\end{equation}
As discussed in \cite[Section 2.2]{Gid18}, whereas existence of solution for the admissible initial value problems 
associated to \eqref{eq:forcebalance} is guaranteed, the same is not true in general for uniqueness.   
The situations in which we may observe multiplicity of solution are however confined to some special 
choices of the friction coefficients, presenting critical symmetries such that the inertial effects, 
usually negligible, become decisive for the evolution of the real system. 
To guarantee uniqueness of solution, we make therefore the following assumption: for every subset of 
indices $J\subseteq\{1,\dots,N\}$ we have
\begin{equation}\label{eq:uniqueness}
\sum_{i\in J} \mu^+_i(t)-\sum_{i\in J^c} \mu^-_i(t)\neq 0 \quad \text{for almost every $t$}
\end{equation}
where $J^c$ is the complement of $J$ \cite[Sec.~5.1]{Gid18}.

%%%%%%%%%%%%%%%%%%%%%%%%%%%%%%%%%%%%%%%%%%%%%%%%%%%%%%%%%%%%%%%%%%%%%%%%%%%%%%%%%%%%%%%%%%%%%%
\subsubsection*{Asymptotic average velocity}

We notice that the information provided by the state $x\in X$ of the crawler can be divided into two components: 
a scalar $y\in Y\cong \R$ describing the \lq\lq average\rq\rq\ position of the crawler, and a $(N-1)$-dimensional 
vector $z\in Z\cong \R^{N-1}$ describing the shape of the crawler.
We therefore define the projections
\begin{equation}
\begin{array}{l}
\pi_Y(x):=\frac{1}{N}\sum_{i=1}^N x_i=:y\in Y\\[2mm]
\pi_Z(x):=(x_2-x_1,\dots,x_N-x_{N-1})=:(z_1,\dots, z_{N-1})\in Z
\end{array}
\end{equation}

We are now ready to state our main result of the section. Let us consider the locomotion $x(t)$ of a crawler, 
guided by the dynamics \eqref{eq:forcebalance}. We define a \emph{gait} 
$\GG=(L_1,\dots,L_{N-1},\mu_1^+,\mu_1^-,\dots,\mu_N^+, \mu_N^-)$ of the crawler as a specific choice of 
$T$-periodic functions  $L_i$, namely, the rest lengths of the actuators, and 
$\mu_j^\pm$, namely, the coefficients in the dissipation potential, satisfying all the running assumptions, 
including \eqref{eq:uniqueness}.

\begin{theorem} \label{th:speed}
	For every gait $\GG$ there exists an average asymptotic velocity $v_0(\GG)$ defined as
	\begin{equation}\label{eq:speed}
	v_0(\GG):=\lim_{t\to+\infty}\frac{\pi_Y(x(t))-\pi_Y(x(t_0))}{t-t_0}
	\end{equation}
	where $x(t)\colon[t_0,+\infty)$ is a solution of \eqref{eq:forcebalance}.
	In particular, $v_0(\GG)$ does not depend on the admissible initial state $x(t_0)\in C(t_0)$, 
meaning that \eqref{eq:speed} holds with the same value of $v_0(\GG)$ for every solution $x(t)$ of \eqref{eq:forcebalance} with gait $\GG$.

Moreover there exists a running-periodic solution $\bar x(t)$ of the form \eqref{eq:runningper} with $\bar v_0=v_0(\GG)$ such that
\begin{equation}\label{eq:crawl_conv}
	\lim_{q\to\infty} \| x(\cdot + q T)-\bar{x}(\cdot)\|_{W^{1,2}([t_0,t_0+T];\mathbb{R}^n)}=0
\end{equation}
(notice that however $\bar x_0$ depends on $x(t_0)$).

\end{theorem}

In order to prove Theorem \ref{th:speed}, we will reformulate 
the evolution equation as a suitable sweeping process, in order to apply the theoretical results 
of Section~\ref{sec:theory}.

%%%%%%%%%%%%%%%%%%%%%%%%%%%%%%%%%%%%%%%%%%%%%%%%%%%%%%%%%%%%%%%%%%%%%%%%%%%%%%%%%%%%%%%%%%%%%%
\subsubsection*{Formulation as sweeping process}

We notice that the elastic energy of our locomotor depends only on the shape component $z=\pi_Z(x)$, since it is invariant 
with respect to rigid translations, namely changes in the  $y=\pi_Y(x)$ component. Indeed, we can write
\begin{equation}\label{eq:linkenergy2}
\EE(t,x)=\frac12 \scal{k \pi_Z(x)-\lsh(t)}{\pi_Z(x)}+\text{time-dependent term}\qquad
\end{equation}
where $ \lsh(t)=\bigl(2kL_1(t),\dots,2kL_{N-1}(t)\bigr)$ and the time-dependent term is
irrelevant for the dynamics \eqref{eq:forcebalance}. 

It is well known \cite{MieThe,MieRou} that the force balance \eqref{eq:forcebalance} can be rewritten equivalently as the variational inequality
\begin{equation}
\scal{k \pi_Z(x(t))-\lsh (t)}{\pi_Z(u-\dot x(t))}+\RR(t,u)-\RR(t,\dot x(t))\geq 0 \quad \text{for every $u\in X$} 
\label{eq:VI}
\end{equation}
From this we deduce straightforwardly that $x(t)$ is a solution of 
\eqref{eq:VI} only if $z(t)=\pi_Z(x(t))$ satisfies:
\begin{equation}
\scal{k z(t)-\lsh (t)}{w-\dot z(t)}+\RRsh (t,w)-\RRsh (t,\dot z(t))\geq 0 \quad \text{for every $w\in Z$}
\label{eq:RVI}
\end{equation}
where $\RRsh $ is the \lq\lq shape-restricted\rq\rq\ dissipation, i.e.~the dissipation after minimization with respect to 
translations of the crawler, defined as
\begin{equation} \label{eq:RRsh}
\RRsh (t,z)=\min \bigl\{\RR(t,x):x\in X, \pi_Z(x)=z\bigr\}
\end{equation}
We remark that $\RRsh$ is well defined, since $\RR$ is convex and coercive.
Moreover, the variational inequality contains the key information on the evolution of the problem, and thus will be 
the main object of our study. 
Indeed, by \cite[Theorem 4.3 and Lemma 5.1]{Gid18}, the assumption \eqref{eq:uniqueness} is sufficient 
to assure uniqueness of solution for 
the initial value problems associated to \eqref{eq:VI}. Moreover, this implies that for almost every 
$t\in[t_0,+\infty)$, namely for all 
the times $t$ for which \eqref{eq:uniqueness} holds, we can  define the function $\vm(t,\cdot)\colon Z\to Y$ as 
the unique one satisfying 
\begin{equation}
\RRsh(t,\pi_Z(x))=\RR(t,x) \qquad \text{if and only if}\qquad  \pi_Y(x)=\vm(\pi_Z(x))  
\end{equation}
The function $\vm$ is  positively homogeneous of degree one in $z$  and every solution $x(t)$ of \eqref{eq:forcebalance} satisfies for almost every $t$
\begin{equation} \label{eq:Yrecover}
	\dot y(t)=\pi_Y(\dot x(t))=\vm(t,\pi_Z(\dot x(t)))=\vm(t,\dot z(t))
\end{equation}
We therefore see that, assuming \eqref{eq:uniqueness}, each solution of \eqref{eq:forcebalance} is easily 
recovered once we solve \eqref{eq:RVI}. Rewriting \eqref{eq:RVI} in the the same form of \eqref{eq:forcebalance} we obtain
\begin{equation} \label{eq:Rfb}
-k z +\lsh(t)\in \partial_{\dot z}\RRsh(t,\dot z)
\end{equation}
The function $\RRsh(t,\cdot)\colon Z\to \R$ is convex  and positively homogeneous of degree one 
\cite[Lemma 2.1]{Gid18}. Let us hence denote by $\RRsh^*(t,\cdot)$ 
the Legendre transform of $\RRsh(t,\cdot)$, by $\partial_{\zeta}\RRsh^*$  the subdifferential in the
second variable, $\Csh(t):=\partial_{\dot z}\RRsh(t,0)$, and by
$\chi_{C}$ the characteristic function  associated to a set $C$. 
Applying the Legendre-Fenchel equivalence to \eqref{eq:Rfb}, 
 we obtain
\begin{align} \label{eq:Rincl}
\dot z\in \partial_{\zeta}\RRsh^*(t,-k z +\lsh(t))&={\partial_{\zeta}}\chi_{\Csh(t)}(-k z +\lsh(t))\notag\\
&=\NN_{\Csh(t)}(-k z +\lsh(t))
\end{align}
where we have set $\Csh(t):=\partial_{\dot z}\RRsh(t,0)$ and the first equality
follows from the fact that $\RR(t,\cdot)$ 
is positively homogeneous of degree one.

  Moreover, recalling that $e_1,\dots,e_N$ denotes the canonical base of $X$, 
by \cite[Lemma 2.2]{Gid18} we deduce  that
\begin{equation}
	\Csh(t)=\{z\in Z : -\mu_i^-(t)\leq\scal{\pi_Z(e_i)}{z}\leq \mu_i^+(t) \quad \text{for $i=1,\dots,N$}\}
\end{equation}
The term $\lsh(t)$ in the last term of \eqref{eq:Rincl} can be included in the changes in the set, leading to
\begin{equation} \label{eq:Rincl2}
\dot z\in \NN_{K(t)}(-k z ) \qquad\text{with $K(t):=\Csh(t)-\lsh(t)$}
\end{equation}
Writing $w:=-k z$, and multiplying  both sides of \eqref{eq:Rincl2} by $-k$, we obtain
\begin{equation}  \label{eq:finite_sweep}
\dot w\in -\NN_{K(t)}(w).
\end{equation}
The set $K(t)$ straightforwardly satisfies \ref{hyp:C1} and \ref{hyp:C2}, with $m=2N$. 
Moreover, also assumption 
\ref{hyp:C3} holds, as a consequence of \eqref{eq:uniqueness}.  Indeed, as we will
discuss later (see Lemma \ref{lem:licqstar} below), one can see that LICQ 
holds for $K(t^*)$ if and only if the inequality in \eqref{eq:uniqueness} is true at $t=t^*$.

\subsubsection*{The function $\vm$}
First of all, we notice that the same transformation used in \eqref{eq:Rincl}, based on the Legendre-Fenchel 
equivalence,  can be applied also to  \eqref{eq:forcebalance}, yielding
\begin{align} \label{eq:incl}
\dot x\in \partial_{\xi}\RR^*(t,-D_x\EE(t,x))&={\partial_{\xi}}\chi_{C(t)}(-D_x\EE(t,x))\notag\\ &=\NN_{C(t)}(-D_x\EE(t,x))
\end{align}
where  $\RR^*(t,\cdot)$ denotes the Legendre transform of $\RR(t,\cdot)$, $\partial_{\xi}$  the subdifferential in the second variable,  and $C(t)$ is defined in \eqref{eq:defC}. 

We highlight that, because of the invariance for translations of the energy $\EE$, the vector $-D_x\EE(t,x)$ is not free to vary in the 
whole set $C(t)$, but it is actually contained in the section $C(t)\cap \pi^{-1}_Y(0)$.

As shown in \cite[Lemma 2.3]{Gid18}, if the inequality in \eqref{eq:uniqueness} holds at time $t=t^*$, then for every $\dot z\in Z$ there exists a unique 
$\dot x\in X$ such that $\dot z=\pi_Z(\dot x)$ and $\dot x\in \NN_{C(t^*)}(\xi)$ for some $\xi\in \pi^{-1}_Y(0)$. 
We prove here a slightly stronger result: at such times $t^*$, the 
decompositions of $\dot x(t^*)\in \NN_{C(t^*)}(x(t^*))$ and $\dot z(t^*)\in \NN_{\Csh(t^*)}(z(t^*))$ 
according to Lemma~\ref{lem:indep_base} are characterized by the very same coefficients $\lambda_1,\dots,\lambda_{2N}$.

Let us take $\xi=(\xi_1,\dots,\xi_N)\in C(t)$ and define the set of active constraints $\JJ(t,\xi)$ as
\begin{equation*}
\JJ(t,\xi)=\bigl\{j\in\{1,\dots,N\}: \xi_j= -\mu_j^-(t) \bigr\}\cup \bigl\{j\in\{N+1,\dots,2N\}: \xi_{j-N}= \mu_{j-N}^+(t) \bigr\}
\end{equation*}
Since the coefficient $\mu_j^\pm$ are always positive, it is easily verified that LICQ holds for $C(t)$ at every time. In particular, we emphasize that at most one constraint within each pair $\{i,i+N\}$ can be active for the same couple $(t,\xi)$.

Hence, by Lemma \ref{lem:indep_base}, we know that, for every  $v\in \NN_{C(t)}(\xi)$ there exists a unique choice of coefficients $\lambda_i\geq 0$, $i=1\dots, 2N$,
\begin{equation}\label{eq:decomp}
v=\sum_{i=1}^N \lambda_i e_i+\sum_{i=N+1}^{2N} -\lambda_i e_{i-N} \qquad\text{where $\lambda_i>0 \Rightarrow i\in \JJ(t,\xi) $}
\end{equation}
The choice of the coefficients $\lambda_i$ does not depend on $\xi$, and the coefficients $\lambda_i$ are Lipschitz continuous with respect to $v$.

Let us set $\nu_i=\pi_Z(e_i)$ for $i=1\dots,N$ and $\nu_i=\pi_Z(-e_{i-N})$ for $i=N+1\dots,2N$. Analogously to $\JJ$, we introduce the set of active constraints $\JJsh(t,\zeta)$ for $\zeta\in\Csh(t)$, so that the vector $\nu_i$, for $i=1,\dots,2N$, is associated to the $i$-th constraint for $\Csh$. We have that for every $t$ and every $\zeta\in\Csh(t)$ holds
\begin{equation}\label{eq:Jrel}
	\JJsh(t,\zeta)=\JJ(t,\xi_\zeta) \qquad\text{where $\xi_\zeta:=\pi^{-1}_Z(\zeta)\cap\pi^{-1}_Y(0)\in C(t)\cap \pi^{-1}_Y(0)$}.
\end{equation}
Let us now consider a vector $w\in Z$. Since $C(t)\cap \pi^{-1}_Y(0)$ is a (non-empty) section of a convex set, there exist $\xi\in \pi^{-1}_Y(0)$ and $v\in \NN_{C(t)}(\xi)$ such that $w=\pi_Z(v)$. By this and \eqref{eq:decomp} we obtain that
\begin{equation}\label{eq:z_decomp}
w=\pi_Z(v)=\sum_{i=1}^{2N} \lambda_i \nu_i,
\end{equation}
where the coefficients $\lambda_i$ are the same obtained in \eqref{eq:z_decomp} for $v$. In particular, $v\in \NN_{C(t)}(\xi)$ for some $\xi\in \pi_Y^{-1}(0)$ if and only if $\pi_Z(v)\in \NN_{\Csh(t)}(\pi_{Z}(\xi))$.

We cannot however tell yet whether the decomposition \eqref{eq:z_decomp} is unique for every $w$, since it may be that there exists multiple suitable vectors $v$ such that $\pi_{Z}(v)=w$. To show this, condition \eqref{eq:uniqueness} comes into play.

\begin{lemma} \label{lem:licqstar}
	\textup{LICQ} holds for $\Csh(t^*)$ if and only if the inequality in \eqref{eq:uniqueness} holds at time $t=t^*$.
\end{lemma}
\begin{proof}
By \eqref{eq:Jrel}, to prove that the inequality in \eqref{eq:uniqueness} implies LICQ for $\Csh(t^*)$, we have to show that for any $\xi\in C(t)\cap \pi^{-1}_Y(0)$ the vectors $\{\nu_i, i\in J(t^*,\xi)\}$ are linearly independent. First, we notice that any $(N-1)$ vectors $\nu_i$ are linearly independent if and only if  at most one vector from each pair $\{\nu_i,\nu_{i+N}\}$ is included. This latter condition is always satisfied when the indices are taken in $\JJ(t^*,\xi)$, as we observed above in this subsection. Hence to prove LICQ it remains to show that, given \eqref{eq:uniqueness} at $t=t^*$, each set $J(t^*,\xi)$ with $\xi\in C(t)\cap \pi^{-1}_Y(0)$ has at most $N-1$ elements. This fact was proved in \cite[Lemma 3.2]{Gid18}.

To show the reverse implication, suppose by contradiction that the inequality in \eqref{eq:uniqueness} does not hold at time $t=t^*$. Again by \cite[Lemma 3.2]{Gid18}, there exists $\xi \in C(t)\cap \pi^{-1}_Y(0)$ such that $\JJ(t^*,\xi)$ has exactly $N$ elements. Since $N$ vectors $\nu_i$ can never be linearly independent, since $\dim Z=N-1$, we deduce that LICQ fails for $\Csh(t^*)$.
\end{proof}
Notice that LICQ for $\Csh(t^*)$ is equivalent to LICQ for $K(t^*)$, since the two sets differ only by a translation.

Using all of the arguments in this subsection and Lemma~\ref{lem:indep_base} we deduce the following fact. 

\begin{prop} \label{prop:lambda_same}
Let $z(t)$ be a solution of \eqref{eq:Rincl}. Suppose that at time $t=t^*$ the inequality in \eqref{eq:uniqueness} holds and $\dot z$ is well defined. Then there exists a unique choice of non-negative coefficients $\lambda_1,\dots\lambda_{2N}$ such that
\begin{equation}\label{eq:dotz_decomp}
\dot z(t^*)=\sum_{i=1}^{2N} \lambda_i \nu_i \qquad\text{where $\lambda_i>0 \Rightarrow i\in \JJsh(t,z(t)) $}
\end{equation}
Moreover, let $x(t)$ be a corresponding solution of the original problem \eqref{eq:forcebalance}, hence satisfying $z(t)=\pi_{Z}(x(t))$. 
Then 
\begin{equation*}
\dot x(t^*)=\sum_{i=1}^N \lambda_i e_i+\sum_{i=N+1}^{2N} -\lambda_i e_{i-N}
\end{equation*}
for the very same coefficients $\lambda_i$.
\end{prop}

This implies that the  decomposition \eqref{eq:dotz_decomp} of $\dot z$ can be used to give an explicit expression for $\vm(t,\cdot)$ at the times where the condition in
\eqref{eq:uniqueness} holds:
\begin{equation}\label{eq:vmdec}
\vm\left(t,\sum_{i=1}^{2N} \lambda_i\nu_i\right)=
\sum_{i=1}^N \frac{\lambda_i}{N}+\sum_{i=N+1}^{2N} -\frac{\lambda_i}{N}
\end{equation}
where we used the fact that $\pi_Y(e_i)=\frac{1}{N}$.
We observe that \eqref{eq:vmdec} in particular implies that $\vm(t,\cdot)$ is Lipschitz continuous. 

We emphasize that the time-dependence of $\vm$ has not disappeared in the right-hand side, but has only become implicit: indeed the coefficients $\lambda_i$ depends not only on the vector $w$, but also at the time $t$ at which the decomposition is computed. That is, at different times the same vector $w$ may have different coefficients in \eqref{eq:z_decomp}. 

\subsubsection*{Proof of Theorem \ref{th:speed}}

By \eqref{eq:finite_sweep} and \eqref{eq:Yrecover}, we deduce that the locomotion problem can be expressed in the form
\begin{equation}\label{eq:sweep_form}
\begin{cases}
\dot w\in -\NN_{K(t)}(w) \\
\dot y= \vm(t,-k^{-1} \dot w)
\end{cases}
\end{equation} 
where we recall that $w=-kz$.

Let us first study the sweeping process in \eqref{eq:sweep_form}. We observe that the set $K(t)$ satisfies \ref{hyp:C1}, \ref{hyp:C2} and, by Lemma \ref{lem:licqstar}, also \ref{hyp:C3}.

We can therefore apply  Theorem \ref{th:attractor} and Corollary \ref{cor:meas_decomp}. Hence, for every admissible 
initial condition $w(t_0)\in K(t_0)$, the corresponding solution $w(t)$ of \eqref{eq:finite_sweep} will converge to a 
periodic solution $\bar w$. Denoting with $\lambda_i(t)$ and $\bar\lambda_i(t)$ their corresponding component of the 
decomposition obtained by Lemma \ref{lem:meas_sel}, using Corollary \ref{cor:meas_decomp} we deduce for 
$i=1,\dots 2N$ that
\begin{equation} \label{eq:conv_doty}
\lim_{q\to+\infty}\norm{ \lambda_i(t+qT)-\bar \lambda_i(t+qT)}_{L^1([t_0,t_0+T],\R)}=0.
\end{equation}
Considering now the second equation in \eqref{eq:sweep_form}, by Proposition \ref{prop:lambda_same} 
and equation \eqref{eq:vmdec} we get
\begin{align}
y(t_2)-y(t_1)&=\int_{t_1}^{t_2}\vm(t,-k^{-1} \dot w)\dd t =\int_{t_1}^{t_2}\vm\left(t,-\frac{1}{k}
    \sum_{i=1}^{2N} \lambda_i(t)\nu_i\right)\dd t \notag\\
&=
-\frac{1}{kN}\sum_{i=1}^N \int_{t_1}^{t_2}\lambda_i(t)\dd t +\frac{1}{kN}\int_{t_1}^{t_2}\sum_{i=N+1}^{2N} \lambda_i(t)\dd t
\end{align}
and, analogously, for the limit periodic solution $\bar w$
\begin{align}
\bar y(t_2)-\bar y(t_1)=
-\frac{1}{kN}\sum_{i=1}^N \int_{t_1}^{t_2}\bar\lambda_i(t)\dd t +\frac{1}{kN}\int_{t_1}^{t_2}\sum_{i=N+1}^{2N} \bar\lambda_i(t)\dd t.
\end{align}
Hence, by Corollary \ref{cor:meas_decomp}, we have
\begin{equation} \label{eq:speed_conv}
\lim_{q\to+\infty}\left(	\frac{y(t_0+qT)-y(t_0+(q-1)T)}{T}-\frac{\bar y(t_0+qT)-\bar y(t_0+(q-1)T)}{T}\right)=0.
\end{equation}
On the other hand, by the periodicity of $\dot{\bar y}$ we have
\begin{equation}\label{eq:periodic_speed}
	\lim_{t\to+\infty}\frac{\bar y(t)-\bar y(t_0)}{t-t_0}=\frac{\bar y(t_0+T)-\bar y(t_0)}{T}=:\bar v_0.
\end{equation}
Hence, combining \eqref{eq:speed_conv} and \eqref{eq:periodic_speed} we get
\begin{equation*} 
\lim_{q\to+\infty}	\frac{y(t_0+qT)-y(t_0+(q-1)T)}{T}=\bar v_0.
\end{equation*}
Since, given a converging sequence $a_n\to\bar a$, the sequence $A_q=\frac{1}{q}\sum_{i=1}^qa_i$ of the arithmetic 
means converges to the same limit $A_q\to \bar a$, we obtain
\begin{equation}
\lim_{t\to+\infty}\frac{y(t)-y(t_0)}{t-t_0}=\bar v_0.
\end{equation}
To conclude the proof, we observe that, as a consequence of equation \eqref{eq:eqder2} in Theorem \ref{th:gudmak}, the value $\bar v_0$ of the corresponding asymptotic velocity is the same for every periodic solution of \eqref{eq:finite_sweep}; 
hence also every non-periodic solution will have the same asymptotic velocity $\bar v_0$.

Noticing that the operator $(\pi_Y,\pi_Z)\colon \R^N\to \R\times \R^{N-1}$ is linear and invertible, by the $W^{1,2}$ asymptotic convergences of $y\to \bar y$ and $w\to \bar w$ we deduce \eqref{eq:crawl_conv}.

%%%%%%%%%%%%%%%%%%%%%%%%%%%%%%%%%%%%%%%%%%%%%%%%%%%%%%%%%%%%%%%%%
\section{Some results on finite-time convergence} \label{sec:finite}

\subsection{Finite-time convergence for a periodic $n$-cell}

Let us first consider the case of a one dimensional sweeping process

\begin{equation}\label{eq:1dim}
\dot x\in -\NN_{[a(t),b(t)]}(x),
\end{equation}
where $a(t)\leq b(t)$ are Lipschitz continuous, $T$-periodic real functions.

\begin{theorem}\label{th:1dim}
Every solution of \eqref{eq:1dim} with initial condition $x(0)\in[a(0),b(0)]$ is $T$-periodic in $[T,+\infty)$.
\end{theorem}
\begin{proof}
Let us make some preliminary observations. Equation \eqref{eq:1dim} has forward-in-time uniqueness of solution; 
this implies that given two solutions $x_1,x_2$ such that $x_1(t_0)\leq x_2(t_0)$ for a certain time $t_0$, we have $x_1(t)\leq x_2(t)$ for every $t\geq t_0$.

By Theorem \ref{th:gudmak}, equation \eqref{eq:1dim} has at least one $T$-periodic solution $\bar x(t)$;
moreover, the set of all $T$-periodic solutions of \eqref{eq:1dim} 
is of the form $\{\bar x+c: c\in [\bar a,\bar b]\}$ 
where $[\bar a,\bar b]$ is a suitable closed interval. 

We claim now that there exist $t_a,t_b\in[0,T)$ such that 
$\bar x(t_a)+\bar a=a(t_a)$ and $\bar x(t_b)+\bar b=b(t_b)$. 

To show this latter statement, we distinguish between two cases:
\begin{itemize}
 \item if $\dot{\bar x}\equiv 0$ and $\bar\epsilon=\max_{t\in[t_0,t_0+T]}b(t)-\bar x(t)-\bar b>0$, then it is easily 
shown that $\bar x(t)+\bar b+\epsilon$ is a $T$-periodic solution for every $0<\epsilon\leq \bar \epsilon$, 
in contradiction with the definition of $\bar b$. The case of the other boundary point is analogous.
 \item if $\dot{\bar x}\not\equiv 0$ then $\bar{x}$ must touch both endpoints
of the interval $[a(t),b(t)]$, since 
$\dot{\bar x}(t)<0$ only if $\bar x(t)=b(t)$, and $\dot{\bar x}(t)>0$ only if $\bar x(t)=a(t)$, hence the
claim is confirmed. 
\end{itemize} 

Let us now consider a generic solution $x(t)$ of \eqref{eq:1dim} with $x(0)\in[a(0),b(0)]$. 
If $x(0)\in[\bar x(0)+\bar a,\bar x(0)+\bar b]$, then $x(t)$ is automatically $T$-periodic for $t\geq 0$.

We now assume that $x(0)>\bar x(0)+\bar b$; the case $x(0)<\bar x(0)+\bar a$ can be treated analogously. By the discussion above we deduce that
\begin{equation*}
b(t_b)\geq x(t_b) \geq \bar x(t_b)+\bar b=b(t_b).
\end{equation*}
By forward uniqueness of solution we deduce that for every $t\geq t_b$ we have  $x(t) = \bar x(t)+\bar b$, hence it is $T$-periodic for $t\ge T$.
\end{proof}

\begin{corollary} \label{cor:cell} Let $C(t)$ be the $n$-cell $T$-periodic in time defined as
	\begin{equation*}
	C(t)=\prod_{i=1}^n [a_i(t),b_i(t)]
	\end{equation*}
where $a_i(t)\leq b_i(t)$ are  Lipschitz continuous,  $T$-periodic real functions. Then every solution of the system
\begin{equation}\label{eq:cube}
	\dot x \in -\NN_{C(t)}(x)
\end{equation}
 with initial condition $x(0)=(x_1(0),\dots,x_n(0))\in C(0)$ is $T$-periodic in $[T,+\infty)$.
\end{corollary}
\begin{proof}
The corollary follows directly by Theorem \ref{th:1dim} by observing that \eqref{eq:cube} is equivalent 
to the $n$ uncoupled one-dimensional sweeping processes
\begin{equation*}
\dot x_i\in -\NN_{[a_i(t),b_i(t)]}(x_i) \qquad \qquad i=1,\dots,n
\end{equation*}
\end{proof}

Theorem \ref{th:1dim} can be applied to our locomotion model in the case $N=1$ of a single segment, corresponding to the case of inching locomotion (cf.~\cite[Sec.~5.2]{Gid18}). Hence with deduce that in this situation, for every admissible gait we have convergence to a periodic behaviour within the first period of actuation.

The case $N>1$ cannot be applied to our models, since $K(t)$ in general is not a Cartesian product of intervals. 

\subsection{First counterexample: acute angle and small movements} \label{ex:small}

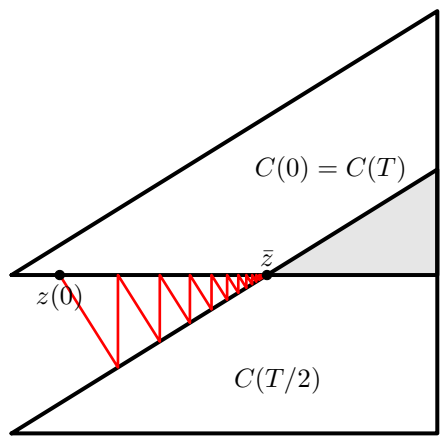
\begin{figure}[h]
	\centering
	\begin{tikzpicture}[line cap=round,line join=round,>=triangle 45,x=1cm,y=1cm, scale=0.7]
		\clip(-2.,-3.55) rectangle (10.,5.5);
		\fill[gray!20] (4.8,0)-- (8,2) -- (8,0)-- (4.8,0);
		\draw [line width=1.5pt] (0,0)-- (8,5)-- (8,0) -- (0,0);
		\draw [line width=1.5pt] (0,-3)-- (8,2) -- (8,-3)-- (0,-3);
		\draw [line width=1pt,red] (0.906,0)-- (2,-1.75)-- (2,0)-- (2.786,-1.259)-- (2.786,0)-- (3.352,-0.905)-- (3.352,0)-- (3.758,-0.651)-- (3.758,0)-- (4.051,-0.468)-- (4.051,0)-- (4.261,-0.337)-- (4.261,0)-- (4.412,-0.2425)-- (4.412,0)-- (4.521,-0.174)-- (4.521,0)-- (4.599,-0.126)-- (4.599,0)-- (4.656,-0.09)-- (4.656,0)-- (4.696,-0.065)-- (4.696,0)-- (4.725,-0.047)-- (4.725,0);
		\draw [fill=black] (0.906,0) circle (2.5pt) node[below] {$z(0)$};
		\draw [fill=black] (4.8,0) circle (2.5pt) node[above] {$\bar z$};
		\draw (6,2) node {$C(0)=C(T)$};
		\draw (5,-2) node {$C(T/2)$};
	\end{tikzpicture}
	\caption{The orbit (red) of a solution $z(t)$ of counterexample \ref{ex:small}, converging only asymptotically to a constant solution $\bar z$. The set $C(t)$ is periodically sliding vertically. Notice  that the solution will stop for progressively longer intervals at the edges of the orbit. The set $\mathcal{Z}$ of periodic solution is the set of constant function with value in the light-gray triangle.}
	\label{fig:exsmall}
\end{figure}

Let us consider the set $C_0\subset \R^2$ defined, for some positive $\alpha$ and $\beta$, as
\begin{equation*}
	C_0=\big\{(z_1,z_2)\in\R^2 : 0\leq z_2\leq \alpha z_1, z_1\leq \beta\big\}.
\end{equation*}
We remark that the constraint $z_1\leq \beta$ is actually irrelevant for our purposes, and it is only added to 
obtain a compact set with a structure compatible with the framework of Section \ref{sec:model}.
We define the $T$-periodic function $l\colon \R\to \R^2$, with $2\beta>T\alpha$, as
\begin{equation}
	l(t)=\begin{cases}
		(0,t) & \text{for $t\in[0,T/2]$}\\
		(0,T/2-t) & \text{for $t\in[T/2,T]$}
	\end{cases}
\end{equation}
and set \begin{equation}
	C(t):=C_0-l(t).
\end{equation}
We now consider the sweeping process 
\begin{equation}
	\dot z\in -\NN_{C(t)}(z)
\end{equation}
and focus on the Poincaré map $\PP_T$ at time $T$, namely the map that assigns to every state in $C_0$ at time $t=0$ 
it corresponding state at time $T$; in particular we focus on the admissible initial points with $z_2(0)=0$. 
Setting $\alpha\gamma:=T/2$ and requiring $\beta>\gamma$, we have
\begin{equation}
	\PP_T((z_1,0))=\begin{cases}
		(\frac{\alpha^2\gamma+z_1}{\alpha^2+1},0)  & \text{for $0\leq z_1<\gamma$}\\
		(z_1,0) & \text{for $\gamma\leq z_1\leq \beta$}
	\end{cases}
\end{equation}
As a consequence, it is easy to see that, for every choice of $0\leq z_1(0)<\gamma$, the sequence $z_1(nT)$ with 
$n\in \N$ is strictly increasing monotone and converges asymptotically to $\gamma$. In other words, all the 
orbits starting at $t=0$ from $(z_1(0),0)$ with $0\leq z_1(0)<\gamma$, are not periodic, but converge asymptotically 
to the periodic orbit starting from $(\gamma, 0)$ at $t=0$. An example is illustrated in Figure \ref{fig:exsmall}.

Notice that this example can be adapted to any situation where a moving set presents an acute angle. Hence it also applies to our locomotion models, for instance in the case of two links with constant anisotropic friction (cf.~\cite{GidDeS17}).
Unfortunately this example would not be particularly relevant in terms of application, since it converges to a constant solution (i.e. shape of the crawler), hence with zero asymptotic average velocity. For a two-link crawler, concretely this corresponds to the forward (or rear) contact point being steady, while the other two make smaller and smaller intermittent adjusting  movements towards a stable configuration.

\subsection{Second counterexample: triangle with wide movements} \label{ex:large}

\begin{figure}[h]
	\centering
	\begin{tikzpicture}[line cap=round,line join=round,>=stealth,x=5cm,y=5cm,line width=1pt,scale=0.5]
		\clip(-1.5,-0.5) rectangle (1.5,2);
		\draw [line width=1.5pt,->] (0.5125998066968301,0.8563839355163796)-- (0.7724074278321621,1.0063839355163795) node[right]{$\nu_2$};
		\draw [line width=1.5pt,->] (0.018166310996777124,0)-- (0.018166310996777124,-0.3) node[right]{$\nu_1$};
		\draw [line width=1.5pt,->] (-0.4762671847032755,0.8563839355163795)-- (-0.7360748058386073,1.0063839355163795) node[left]{$\nu_3$};
		\draw [line width=1.5pt] (0.018166310996777623,1.712767871032759)-- (-0.9707006804033285,0)node[left]{$P$};
		\draw [line width=1.5pt] (-0.9707006804033285,0)-- (1.0070333023968827,0)node[right]{$Q$};
		\draw [line width=1.5pt] (1.0070333023968827,0)-- (0.018166310996777623,1.712767871032759)node[above]{$R$};
		\draw [->,red] (0.35178097425805393,0.9541757489324312)-- (0.351780974258054,0);
		\draw [->,red] (0.351780974258054,0)-- (0.8432202203621757,0.28373258102856125);
		\draw [->,red] (0.8432202203621757,0.28373258102856125)-- (-0.39436064368592183,0.99825022603066);
		\draw [->,red] (-0.39436064368592183,0.99825022603066)-- (-0.39436064368592183,0);
		\draw [->,red] (-0.39436064368592183,0)-- (0.6566848158761818,0.6068213790087141);
		\draw [->,red] (0.6566848158761818,0.6068213790087141)-- (-0.30109294144292475,1.159794625020737);
		\draw [->,red] (-0.30109294144292475,1.159794625020737)-- (-0.3010929414429248,0);
		\draw [->,red] (-0.3010929414429248,0)-- (0.6800017414369311,0.5664352792611951);
		\draw [->,red] (0.6800017414369311,0.5664352792611951)-- (-0.3127514042232994,1.139601575146977);
		\draw [->,red] (-0.3127514042232994,1.139601575146977)-- (-0.31275140422329944,0);
		\draw [->,red] (-0.31275140422329944,0)-- (0.6770871257418373,0.5714835417296349);
		\draw [->,red] (0.6770871257418373,0.5714835417296349)-- (-0.31129409637575267,1.1421257063811971);
		\draw [fill=black] (0.35178097425805393,0.9541757489324312) circle (1.5pt) node[above left] {$x(0)$};
	\end{tikzpicture}
	\caption{The orbit (red) of a solution $z(t)$ of counterexample \ref{ex:large}, converging only asymptotically to a periodic solution.}
	\label{fig:equilateral}
\end{figure}
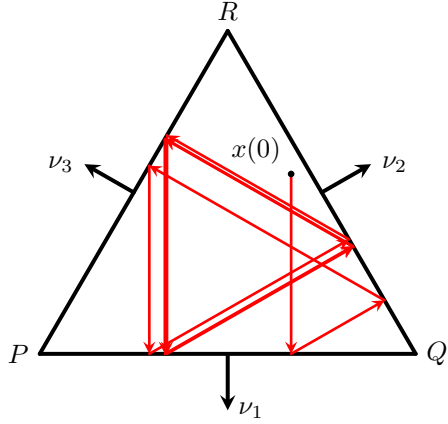

Let us consider the equilateral triangle $\widehat{PQR}$ shown in Figure \ref{fig:equilateral}. 
We denote with $\nu_1,\nu_2,\nu_3$ the unit normal vectors to the edges $\overline{PQ}$, $\overline{QR}$ and $\overline{RP}$.
We introduce the $T$-periodic function $f\colon\R\to \R^2$ defined by 
\begin{equation*}
f(t)=
\begin{cases}\alpha \nu_1 &\text{for $0\leq t<\frac{T}{3}$}\\
\alpha \nu_2 &\text{for $\frac{T}{3}\leq t<\frac{2T}{3}$}\\
\alpha \nu_3 &\text{for $\frac{2T}{3}\leq t<T$}
\end{cases}
\end{equation*}
and extended by periodicity. We remark $f$ has zero average on each period since $\nu_1+\nu_2+\nu_3=0$, hence any of its primitive $F(t)$ is $T$-periodic.
We consider the perturbed sweeping process
\begin{equation} \label{eq:equilateral}
\dot x\in -\NN_{\widehat{PQR}}(x)+f(t)
\end{equation}
which corresponds, up to a periodic change of coordinates, to the classical sweeping process with moving set $C(t)=\widehat{PQR}-F(t)$.
It is easily shown that, for $\alpha T$ sufficiently large, the dynamics \eqref{eq:equilateral} sends every 
initial state $x(0)$ to the segment $\overline{PQ}$ within the first third of the period. Then in the second third 
of the period all the orbits are sent to the segment $\overline{QR}$, which in turn is sent to the segment $\overline{RP}$ in the last third of the period.
Hence, to characterize the dynamics for $t>T/3$ it is sufficient to look at such projection from an edge to the following one. 
Parametrizing each segment as $[0,1]$, such projection is given by the map $\phi\colon[0,1]\to [0,1]$ defined as
\begin{equation*}
\phi(z)=\frac{1-z}{2}
\end{equation*}
Hence the Poincaré map with respect to a period, sending the segment $\overline{PQ}$ to itself is given by $\phi^3$. 
We observe that this map has a single fixed point, corresponding to a single periodic solution of \eqref{eq:equilateral}, 
whereas all the other orbits of $\phi^3$ converge asymptotically to such fixed point, but not in finite time.

We remark that the same construction still holds if we consider a general acute  triangle $\widehat{PQR}$, obviously with three different projection $\phi_i$. However, the idea is no longer applicable if the triangle is rectangle or obtuse: it can be easily verified that in such situation we obtain instead, always for large $\alpha T$, convergence within the first period. This latter case seems to be the one occurring with the set $K(t)$ in our models.

%%%%%%%%%%%%%%%%%%%%%%%%%%%%%%%%%%%%%%%%%%%%%%%%%%%%%%%%%%%%%%%%%%%
%%%%%%%%%%%%%%%%%%%%%%%%%%%%%%%%%%%%%%%%%%%%%%%%%%%%


\begin{thebibliography}{30}
	\bibitem{And} J. Andres, Nonlinear rotations, Nonlinear Analysis: Theory, Methods \& Applications 30 (1997), 495--503.
	\bibitem{AnBePa}  J. Andres, D. Bednařík and K. Pastor, On the notion of derivo-periodicity. J. Math. Anal. Appl. 303 (2005),  405-417.
\bibitem{ColGid} G. Colombo and P. Gidoni, On the optimal control of rate-independent soft crawlers, Journal de Mathématiques Pures et Appliquées 146 (2021), 127–157.
\bibitem{CT} G. Colombo and L. Thibault, Prox-regular sets and applications, in  
\textit{Handbook of nonconvex analysis and applications}, 99-182, D. Y. Gao and D. Motreanu eds., Int. Press (2010).
\bibitem{EldJac}  J. Eldering and H.O. Jacobs, The role of symmetry and dissipation in biolocomotion, SIAM J. Appl. Dyn. Syst. 15 (2016), 24--59.
\bibitem{FaPaZo}  F. Fassò, S. Passarella, and M. Zoppello,
Control of locomotion systems and dynamics in relative periodic orbits
Journal of Geometric Mechanics 12 (2020), 395--420.
\bibitem{Gid18} P. Gidoni, Rate-independent soft crawlers, Quarterly Journal of Mechanics and Applied Mathematics 71 (2018), 369--409.
\bibitem{GidDeS17} P. Gidoni and A. DeSimone, Stasis domains and slip surfaces in the locomotion of 
a bio-inspired two-segment crawler, Meccanica 52 (2017), 587--601.
\bibitem{GidDeS18} P. Gidoni and A. DeSimone, On the genesis of directional friction through bristle-like 
mediating elements crawler, ESAIM: Control, Optimization and Calculus of Variations 23 (2017), 1023--1046.
\bibitem{GidRiv} P. Gidoni and F. Riva, A vanishing inertia analysis for finite dimensional rate-independent systems with nonautonomous dissipation and an application to soft crawlers, preprint arXiv:2007.09069.
\bibitem{GuMaRa} I. Gudoshnikov,  O. Makarenkov and D. Rachinskiy, Finite-time stability of polyhedral sweeping processes with application to elastoplastic systems, preprint arXiv:2011.07744.
\bibitem{GudMak} I. Gudoshnikov and O. Makarenkov, Stabilization of the response of cyclically loaded lattice spring models with plasticity, ESAIM: Control, Optimization and Calculus of Variations, in print, DOI https://doi.org/10.1051/cocv/2020043 
\bibitem{GKMV} I. Gudoshnikov, M. Kamenskii, O. Makarenkov, and N. Voskovskaia. One-period stability analysis 
of polygonal sweeping processes with application to an elastoplastic model, Math. Model. Nat. Phenom. 15 (2020),
25 pp.
\bibitem{HoWi} D.G.E. Hobbelen and M. Wisse, Limit Cycle Walking. In \textit{Humanoid Robots, Human-like Machines}, edited by M. Hackel, I-Tech Education and Publishing, 2007.
\bibitem{KeMu} S.D. Kelly and R.M. Murray, Geometric phases and robotic locomotion. Journal of Robotic Systems 12 (1995), 417-431.
\bibitem{Kre} P. Krej\v{c}\'\i, \textit{Hysteresis, Convexity and Dissipation in Hyperbolic Equations}, Gattotoscho, 1996.
\bibitem{LeHoMi} M. Levi,  F.C. Hoppensteadt, W.L. Miranker, Dynamics of the Josephson junction. Quart. Appl. Math. 36 (1978), 167--198.
\bibitem{Mak} O. Makarenkov, Existence and stability of limit cycles in the model of a planar passive biped walking down a slope, Proc. R. Soc. A. 476 (2020), 20190450.
\bibitem{Mar}  R. Martins,  The attractor of an equation of Tricomi's type, J. Math. Anal. Appl. 342 (2008), 1265--1270.
\bibitem{MieThe} A. Mielke and F. Theil. On rate-independent hysteresis models.
NoDEA. Nonlinear Differential Equations and Applications 11 (2004), 151--189.
\bibitem{MieRou} A. Mielke and T. Roub\'\i\v{c}ek, \textit{Rate-independent systems. Theory
and application}, Springer, New York, 2015.
\bibitem{PoFeTa} B. Pollard, V. Fedonyuk, and P. Tallapragada, Swimming on limit cycles with nonholonomic constraints, Nonlinear Dyn 97 (2019), 2453--2468.
\end{thebibliography}
\end{document}